\documentclass[10pt]{amsart}
\usepackage[all]{xy}
\usepackage{a4wide}
\usepackage{amssymb}
\usepackage{amsthm,xcolor}
\usepackage{hyperref} 
\usepackage{graphicx}
\hypersetup{colorlinks=true,linkcolor=blue,citecolor=magenta}
\usepackage{amsmath}
\usepackage{amscd,enumitem}
\usepackage{verbatim}
\usepackage{float}
\usepackage{caption}
\usepackage{color}
\usepackage[mathscr]{eucal}
\usepackage[all]{xy}
\usepackage{hyperref}
\usepackage{bbm}
\usepackage{pgfplots}
\usepackage{tikz}
\pgfplotsset{compat=1.17}
\usetikzlibrary{quotes,angles}
\usepackage{dcolumn}
\def\={\;=\;}  \def\+{\,+\,}

\theoremstyle{plain}
\newtheorem{theorem}{Theorem}

\newtheorem*{corollary*}{Corollary}
\newtheorem*{Example*}{Example}

\newtheorem{conjecture}{Conjecture}

\newtheorem*{conjectureold}{Conjecture}
\theoremstyle{definition}

\newtheorem*{def*}{Definition}
\newtheorem*{theorem*}{Theorem}
\newtheorem*{example}{Example}

\newtheorem*{problem}{Problem}

\newtheorem*{definition*}{Definition}

\theoremstyle{remark}
\newtheorem*{remark}{Remark}

\newcommand{\Z}{\mathbb{Z}}

\title[Do perfect powers repel partition numbers?]{Do Perfect powers repel partition numbers?}

\author{Mircea Merca, Ken Ono and Wei-Lun Tsai}
\dedicatory{In memory of H\"aim Brezis and his contributions to mathematics}

\address{Department of Mathematical Models and Methods, Fundamental Sciences Applied in Engineering Research Center, National University of Science and Technology Politehnica Bucharest, Bucharest, 060042, Romania}
\email{mircea.merca@upb.ro}

\address{Department of Mathematics, University of Virginia, Charlottesville, VA 22904, USA}
 \email{ken.ono691@virginia.edu}
 
 \address{Department of Mathematics, University of South Carolina, Columbia, SC 29208, USA}
 \email{weilun@mailbox.sc.edu}


\begin{document}
\thanks{The second
  author thanks  the NSF
(DMS-2002265 and DMS-2055118) for their support, and the third author acknowledges the support of an AMS-Simons Travel Grant. }
 
\begin{abstract}In 2013 Zhi-Wei Sun conjectured that $p(n)$ is never a power of an integer when $n>1.$ We confirm this claim in many cases.  We also observe that integral powers appear to repel the partition numbers. If $k>1$ and $\Delta_k(n)$ is the distance between $p(n)$ and the nearest $k$th power, then for every $d\geq 0$ we conjecture that there are at most finitely many $n$ for which $\Delta_k(n)\leq d.$ More precisely, for every $\varepsilon>0,$ we conjecture that
 $$M_k(d):=\max\{n \ : \ \Delta_k(n)\leq d\}=o( d^{\varepsilon}).
 $$
In $k$-power aspect with $d$ fixed, we also conjecture that if $k$ is sufficiently large, then
$$
M_k(d)=\max \left\{ n \ : \ p(n)-1\leq d\right\}.
$$
In other words,  $1$ generally  appears to be the closest $k$th power among the partition numbers.
\end{abstract}

\maketitle

\section*{Introduction}
A {\it partition} of a non-negative integer $n$ is any non-increasing sequence of positive integers that sum to $n$,  and the {\it partition function} $p(n)$ denotes their number. The long history of $p(n)$ (for example, see \cite{Andrews}) is marked with celebrated contributions by mathematicians such as Euler, Hardy, and Ramanujan. Indeed, we have Euler's
 famous recurrence relation, which asserts for positive $n$ that
$$
p(n)=\sum_{k\in \Z\setminus \{0\}} (-1)^{k+1} p\left(n-\frac{3k^2+k}{2}\right)=p(n-1)+p(n-2)-p(n-5)-p(n-7)+\dots.
$$
Hardy and Ramanujan famously proved the asymptotic formula
$$
p(n)\sim \frac{1}{4n\sqrt{3}}\cdot e^{\pi\sqrt{2n/3}},
$$
in work which gave birth to the ``circle method'', a tool which is now ubiquitous in analytic number theory.
Their approach was later perfected by Rademacher, who obtained an exact formula as an infinite convergent series
$$
p(n)= 2 \pi (24n-1)^{-\frac{3}{4}} \sum_{k =1}^{\infty} \frac{A_k(n)}{k}\cdot
I_{\frac{3}{2}}\left( \frac{\pi \sqrt{24n-1}}{6k}\right).
$$
Here $I_{\frac{3}{2}}(\cdot)$ is a  modified Bessel function of the first kind, and $A_k(n)$ is a Kloosterman-type sum.
Extending beyond the world asymptotics, exact formulas, and recurrence relations,  Ramanujan famously proved that
\begin{displaymath}
\begin{split}
p(5n+4)&\equiv 0\pmod 5,\\
p(7n+5)&\equiv 0\pmod 7,\\
p(11n+6)&\equiv 0\pmod{11},
\end{split}
\end{displaymath}
representing the first cases of three infinite families of congruences modulo arbitrary powers of 5, 7, and 11. These congruences have inspired an entire field of study in combinatorics and the theory of modular forms.

\section*{Sun's Conjecture}

Despite its long history, the partition function continues to be a source of beautiful and fascinating problems. Here we refine one such problem, the following 2013 conjecture of Zhi-Wei Sun (see Conjecture 8.9 (iii) of \cite{Sun1} and also  \cite{Sun2, Sun3}).\footnote{In the case of  squares, T. Amdeberhan independently made the same conjecture/speculation in 2017 \cite{Amdeberhan}.}

\begin{conjectureold} [Sun] If $n>1$, then $p(n)$ is not a $k$th power of an integer for any $k>1.$
\end{conjectureold}

Although we are unable to prove this conjecture, we show that recent deep work of Bennett and Siksek \cite{BS} on Diophantine equations offers strong evidence supporting its truth. Their work precludes many partition numbers from being perfect powers. To make this precise, we define the following special set of positive integers
$$
S:=\left\{ n \ : \ p(n)=x^2 + \ell^a,\ {\text {where}}\ x\in \Z, \ 2\leq \ell <100\ {\text{prime,}}\  a\geq 1, {\text {and}} \ \ell \nmid x\right\}.
$$

\begin{theorem} For integers $n\in S,$ the following are true.

\noindent
(1) We have that $p(n)$ is not a $k$th power for any $k\geq 3.$

\noindent
(2) Sun's Conjecture for squares implies the conjecture for all $k$th powers, where $k>1$.
\end{theorem}

\begin{proof}[Proof of Theorem~1]  Bennett and Siksek consider Diophantine equations of the form
$$
x^2+q^{\alpha}=y^k.
$$
They prove (see Theorem~1 of \cite{BS}) 
 that if $x, y, q, \alpha$ and $k$ are solutions in the positive integers with
$2\leq q<100$ prime, $q\nmid x$, $k\geq 3$, then $(q, \alpha, y, k)$ is a member of an explicit finite list
$$
\left \{ (2,1,3,3), (2,2,5,3), (2,5,3,4),\dots, (89,1,5,3), (97,2,12545, 3), (97,1,7,4)\right\}.
$$
Applied to our setting, we replace $q$ by $\ell$, and note that if $n\in S$, then we have the identity
$$
    p(n)=x^2+\ell^a.
$$
Therefore, if $p(n)=y^k$ is a $k$th power, then apart from the finitely many possible candidates given by this explicit finite list, it must be that $k\in \{1, 2\}.$ By brute force, we checked that none of the numbers $y^k$ from this list are actual partition numbers. This is an easy finite calculation, as $p(n)$ is an increasing sequence in $n.$ Therefore, both claims follow.
\end{proof}

\begin{example} Using Theorem~1 (1), we show that 
 $p(3), p(4), \dots, p(15)=176$ are not $k$th powers for any $k\geq 3$. The point is that this conclusion is made without factorizing these explicit values.
To this end, we list the relevant numbers of the form $x^2+\ell^a$ that are used to define $S$. Ordered by size, they are
\begin{displaymath}
\begin{split}
&\left \{3, 4, 5, \dots, 36\right \} \cup \left \{ 38,  39, 40, \dots, 63\right \} \cup \left \{65, 66, 67,\dots, 120\right\} \cup
\left \{122, 123, 124, \dots, 135\right \} 
 \cup \left\{137, 138\right\}\\
 &\hskip1.75in\ \ \ \ \  \cup \left \{140, 141, \dots, 155\right \} \cup
\left \{157, 158, \dots, 164\right\} \cup \left\{167, 168,\dots, 183\right\}  \cup \dots.
\end{split}
\end{displaymath}
The positive integers not exceeding $p(15)=176$ that are missed are: 1, 2, 37, 64, 121, 136, 139, 156, 165, 166.
 Since the partition numbers
 $p(3)=3, p(4)=5, \dots, p(15)=176$ are not among these ten values, Theorem~1 (1) implies that none are $k$th powers for any $k\geq 3.$ 
 
 \end{example}

\begin{remark} Theorem~1 does not apply to all $n$. Indeed, among the first 20 partition numbers, it does not apply for the three values $p(2)=2, p(16)=231$, and $p(19)=490.$ There are no positive integer solutions to
$$
x^2+\ell^a \in  \left\{2, 231, 490\right\},
$$
where $\ell\leq 100$ is prime and $\ell\nmid x.$

\end{remark}

\section*{Further Conjectures: Do $k$th powers repel partition numbers?}

To refine Sun's Conjecture, it makes sense to search for ``near misses'', partition numbers that are close to $k$th powers.
A brief scan of a table of values of $p(n)$ does not reveal many near misses.  In fact, the partition numbers appear to be repelled by the $k$th powers. 
Indeed, the perfect squares up to 250 are:
$$
1, \ 4, \ 9, \ 16, \ 25, \ 36, \ 49, \ 64, \ 81, \  {\bf100}, \ 121, \ 144, \ 169, \ 196, \ 225, \ \dots
$$
while the partition numbers in this range are:
$$
1,\ 1, \ 2, \ 3, \ 5, \ 7, \ 11, \ 15, \ 22, \ 30, \ 42, \ 56, \ 77, \  {\bf101}, \ 135, \ 176, \ 231,\dots.
$$
Apart from the perfect square 1, the only notable near miss is the number $p(13)=101$.

To quantify this observation,
we define $\Delta_k(n)\geq 0$ to be the distance between $p(n)$ and its nearest $k$th power. Namely, we let
\begin{equation}
\Delta_k(n):=\min\left\{ |p(n)-m^k| \ : \ m\in \Z\right\}.
\end{equation}
Table 1 includes these values for squares, cubes, and fourth powers for 
$$
p(10)=42,\ p(20)=627, \ p(30)=5604, \ p(40)=37338,  \ {\text {\rm and}}\ \ p(50)=204226.
$$

\smallskip

	\begin{table}[H]
		\begin{tabular}{|c|l|l|l|} \hline
			$n$ & $\Delta_2(n)$ & $\Delta_3(n)$ & $\Delta_4(n)$ \\ \hline
			10& \ \ 6 & 15 & 26  \\ \hline
			20 & \ \ 2 & 102 & 2  \\ \hline
			30 &\ \ 21 & 228 & 957 \\ \hline
			40 &\ \ 89   & 1401 & 1078 \\ \hline
			50 & \ \  78 & 1153 & 9745 \\ \hline
		\end{tabular}\smallskip
		\caption{Sample values of $\Delta_k(n)$}
			\end{table}

\smallskip

\noindent
Based on numerics performed on a computer,  we make the following conjecture.

\begin{conjecture} If $k>1$ and $d\geq 0,$ then there are at most finitely many $n$ for which $$\Delta_k(n)\leq d.$$
\end{conjecture}

Assuming this conjecture, we study the ``growth'' of $\Delta_k(n).$ To this end,  we define
\begin{equation}
M_k(d):=\max \left \{ n \ : \ \Delta_k(n)\leq d\right\},
\end{equation}
the last $n$ for which $p(n)$ is within $d$ of a $k$th power. Table~2 gives conjectured values for $d\leq 10^{70}.$
\smallskip

\begin{table}[H]
	\begin{tabular}{|c|c|c|c|c|c|c|c|c|c|c|}\hline
		$d$ & $M_2(d)$ & $M_3(d)$  & $M_4(d)$ & $M_5(d)$ & $M_6(d)$ & $M_7(d)$ & $M_8(d)$ & $M_{50}(d)$ & $M_{100}(d)$ \\ \hline
	$0$ & $1$ & $1$ & $1$ & $1$ & $1$ & $1$ & $1$ & $1$ & $1$ \\ \hline
	$10^0$ & $ 35$ & $ 5$ & $ 7$ & $ 2$ & $ 2$ & $ 2$ & $ 2$ & $ 2$ & $ 2$ \\ \hline
	$10^5$ & $ 201$ & $ 133$ & $ 87$ & $ 82$ & $ 64$ & $ 71$ & $ 64$ & $ 45$ & $ 45$ \\ \hline
	$10^{10}$ & $ 527$ & $ 295$ & $ 265$ & $ 247$ & $ 227$ & $ 258$ & $ 196$ & $ 135$ & $ 135$ \\ \hline
	$10^{15}$ & $ 1100$ & $ 705$ & $ 482$ & $ 454$ & $ 445$ & $ 388$ & $ 387$ & $ 279$ & $ 269$ \\ \hline
	$10^{20}$ & $ 2058$ & $ 1019$ & $ 806$ & $ 745$ & $ 654$ & $ 653$ & $ 633$ & $ 444$ & $ 444$ \\ \hline
	$10^{25}$ & $ 2595$ & $ 1525$ & $ 1203$ & $ 1052$ & $ 971$ & $ 978$ & $ 890$ & $ 663$ & $ 662$ \\ \hline
	$10^{30}$ & $ 3804$ & $ 2135$ & $ 1636$ & $ 1564$ & $ 1337$ & $ 1244$ & $ 1280$ & $ 941$ & $ 941$ \\ \hline
	$10^{35}$ & $ 5030$ & $ 2815$ & $ 2444$ & $ 1930$ & $ 1886$ & $ 1747$ & $ 1620$ & $ 1239$ & $ 1221$ \\ \hline
	$10^{40}$ & $ 6340$ & $ 3714$ & $ 2849$ & $ 2513$ & $ 2366$ & $ 2246$ & $ 2047$ & $ 1565$ & $ 1562$ \\ \hline
	$10^{45}$ & $ 8253$ & $ 4424$ & $ 3516$ & $ 3178$ & $ 2866$ & $ 2754$ & $ 2685$ & $ 2170$ & $ 2170$ \\ \hline
	$10^{50}$ & $ 9646 $ & $ 5479$ & $ 4314$ & $ 3726$ & $ 3537$ & $ 3411$ & $ 3134$ & $ 2556$ & $ 2368$ \\ \hline
	$10^{55}$ & $ 11524$ & $ 6808$ & $ 5229$ & $ 4802$ & $ 4169$ & $ 3933$ & $ 3815$ & $ 2909$ & $ 2833$ \\ \hline
	$10^{60}$ & $ 13723$ & $ 8088$ & $ 6117$ & $ 5318$ & $ 4854$ & $ 4629$ & $ 4506$ & $ 3501$ & $ 3382$ \\ \hline
	$10^{65}$ & $ 15516$ & $ 8924$ & $ 6961$ & $ 6403$ & $ 5676$ & $ 5502$ & $ 5232$ & $ 4129$ & $ 3884$ \\ \hline
	$10^{70}$ & $ 18237$ & $ 10252$ & $ 8084$ & $ 7056$ & $ 6628$ & $ 6268$ & $ 6098$ & $ 4665$ & $ 4502$ \\ \hline
	\end{tabular}\smallskip	
	\caption{Conjectured values of $M_k(d)$}
\end{table}
\normalsize

Based on these numerics, we make the following refinement of Conjecture 1.

\begin{conjecture} If $k>1,$ then the following are true.

\noindent
(1) We have that $M_k(0)=1.$

\noindent
(2) For every $\varepsilon>0,$ we have that $M_k(d)=o(d^{\varepsilon})$.
\end{conjecture}

\begin{remark}
Conjecture 2 (1) is a recapitulation of Sun's Conjecture, and claim (2) asserts that
$$
\lim_{d\rightarrow +\infty} \frac{M_k(d)}{d^{\varepsilon}}=0.
$$
\end{remark}

Conjecture 2 asserts that the $M_k(d)$ numbers are bounded by any positive power of $d$.
It is natural to seek lower bounds. In Figure 1, we offer the values of $M_k(d)$ for $d\leq 10^{70}$ in the form of a graph to illustrate the slow growth of these values.

\medskip

\begin{figure}[H]
\begin{tikzpicture}
    \begin{axis}[
        width=14.5cm,
        height=7cm,
        xlabel={$i$ (where $d = 10^i$)},
        ylabel={\ \ },
        legend pos=north west,
        grid=major,
        scaled ticks=false,
    ]

    % quintic log model prediction
  
    % M_k(d) values
    \addplot[color=red, thick, mark=star] coordinates {
(0,35) (1,35) (2,87) (3,143) (4,148) (5,201) (6,280) (7,312) (8,473) (9,483) (10,527) (11,775) (12,815) (13,851) (14,1020) (15,1100) (16,1352) (17,1352) (18,1505) (19,1675) (20,2058) (21,2084) (22,2361) (23,2361) (24,2487) (25,2595) (26,3101) (27,3238) (28,3479) (29,3517) (30,3804) (31,3947) (32,4047) (33,4412) (34,4582) (35,5030) (36,5253) (37,5479) (38,5822) (39,6032) (40,6340) (41,6576) (42,6964) (43,7210) (44,7530) (45,8253) (46,8253) (47,8615) (48,9021) (49,9475) (50,9646) (51,10013) (52,10932) (53,10932) (54,11477) (55,11524) (56,12000) (57,12623) (58,12623) (59,13309) (60,13723) (61,14329) (62,14352) (63,14736) (64,15516) (65,15516) (66,16342) (67,17239) (68,17239) (69,17399) (70,18237)

     % (0, 35) (1, 35) (2, 87) (3, 143) (4, 148) (5, 201) (6, 280) (7, 312) (8, 473) (9, 483) (10, 527) (11, 775) (12, 815) (13, 851) (14, 1020) (15, 1100) (16, 1352) (17, 1352) (18, 1505) (19, 1675) (20, 2058) (21, 2084) (22, 2361) (23, 2361) (24, 2487) (25, 2595) (26, 3101) (27, 3238) (28, 3479) (29, 3517) (30, 3804) (31, 3947) (32, 4047) (33, 4412) (34, 4582) (35, 5030) (36, 5253) (37, 5479) (38, 5822) (39, 6032) (40, 6340) (45, 8253) (50, 9646) (55, 11524)
        % (0, 35) (1, 35) (2, 87) (3, 143) (4, 148) (5, 201) (6, 280) (7, 312) (8, 473) (9, 483) (10, 527) (11, 775) (12, 815) (13, 851) (14, 1020) (15, 1100) (16, 1352) (17, 1352) (18, 1505) (19, 1675) (20, 2058)
    };
    \addlegendentry{$M_2(d)$}

    \addplot[color=green, thick, mark=star] coordinates {

(0,5) (1,22) (2,57) (3,57) (4,75) (5,133) (6,146) (7,229) (8,261) (9,270) (10,295) (11,361) (12,505) (13,505) (14,592) (15,705) (16,709) (17,811) (18,914) (19,995) (20,1019) (21,1194) (22,1294) (23,1294) (24,1376) (25,1525) (26,1693) (27,1812) (28,1848) (29,2032) (30,2135) (31,2364) (32,2566) (33,2579) (34,2661) (35,2815) (36,3035) (37,3102) (38,3315) (39,3616) (40,3714) (41,3800) (42,3851) (43,4214) (44,4424) (45,4424) (46,4786) (47,4931) (48,4931) (49,5105) (50,5479) (51,5779) (52,6023) (53,6093) (54,6327) (55,6808) (56,7005) (57,7005) (58,7184) (59,7440) (60,8088) (61,8102) (62,8499) (63,8650) (64,8650) (65,8924) (66,9292) (67,9413) (68,9870) (69,9907) (70,10252)

% (0, 5) (1, 22) (2, 57) (3, 57) (4, 75) (5, 133) (6, 146) (7, 229) (8, 261) (9, 270) (10, 295) (11, 361) (12, 505) (13, 505) (14, 592) (15, 705) (16, 709) (17, 811) (18, 914) (19, 995) (20, 1019) (21, 1194) (22, 1294) (23, 1294) (24, 1376) (25, 1525) (26, 1693) (27, 1812) (28, 1848) (29, 2032) (30, 2135) (31, 2364) (32, 2566) (33, 2579) (34, 2661) (35, 2815) (36, 3035) (37, 3102) (38, 3315) (39, 3616) (40, 3714) 

        % (0, 5) (1, 22) (2, 57) (3, 57) (4, 75) (5, 133) (6, 146) (7, 229) (8, 261) (9, 270) (10, 295) (11, 361) (12, 505) (13, 505) (14, 592) (15, 705) (16, 709) (17, 811) (18, 914) (19, 995) (20, 1019)
    };
    \addlegendentry{$M_3(d)$}

    \addplot[color=black, thick, mark=star] coordinates {

(0,7) (1,20) (2,26) (3,57) (4,79) (5,87) (6,110) (7,170) (8,170) (9,265) (10,265) (11,287) (12,420) (13,420) (14,463) (15,482) (16,571) (17,627) (18,649) (19,773) (20,806) (21,938) (22,952) (23,1120) (24,1174) (25,1203) (26,1325) (27,1470) (28,1534) (29,1534) (30,1636) (31,1789) (32,1973) (33,1978) (34,2162) (35,2444) (36,2444) (37,2444) (38,2582) (39,2849) (40,2849) (41,3001) (42,3107) (43,3320) (44,3387) (45,3516) (46,3795) (47,3879) (48,3952) (49,4250) (50,4314) (51,4525) (52,4595) (53,4943) (54,4943) (55,5229) (56,5299) (57,5511) (58,6078) (59,6078) (60,6117) (61,6192) (62,6395) (63,6619) (64,6752) (65,6961) (66,7248) (67,8084) (68,8084) (69,8084) (70,8084)

    % (0, 7) (1, 20) (2, 26) (3, 57) (4, 79) (5, 87) (6, 110) (7, 170) (8, 170) (9, 265) (10, 265) (11, 287) (12, 420) (13, 420) (14, 463) (15, 482) (16, 571) (17, 627) (18, 649) (19, 773) (20, 806) (21, 938) (22, 952) (23, 1120) (24, 1174) (25, 1203) (26, 1325) (27, 1470) (28, 1534) (29, 1534) (30, 1636) (31, 1789) (32, 1973) (33, 1978) (34, 2162) (35, 2444) (36, 2444) (37, 2444) (38, 2582) (39, 2849) (40, 2849)

        % (0, 7) (1, 20) (2, 26) (3, 57) (4, 79) (5, 87) (6, 110) (7, 170) (8, 170) (9, 265) (10, 265) (11, 287) (12, 420) (13, 420) (14, 463) (15, 482) (16, 571) (17, 627) (18, 649) (19, 773) (20, 806)
    };
    \addlegendentry{$M_4(d)$}

    \addplot[color=yellow, thick, mark=star] coordinates {

(0,2) (1,10) (2,22) (3,32) (4,82) (5,82) (6,87) (7,135) (8,170) (9,200) (10,247) (11,277) (12,361) (13,387) (14,387) (15,454) (16,521) (17,637) (18,637) (19,706) (20,745) (21,796) (22,875) (23,919) (24,973) (25,1052) (26,1127) (27,1205) (28,1269) (29,1415) (30,1564) (31,1613) (32,1733) (33,1898) (34,1898) (35,1930) (36,2121) (37,2125) (38,2281) (39,2466) (40,2513) (41,2613) (42,2807) (43,2841) (44,3023) (45,3178) (46,3313) (47,3365) (48,3671) (49,3671) (50,3726) (51,3907) (52,4100) (53,4159) (54,4386) (55,4802) (56,4802) (57,4802) (58,5054) (59,5155) (60,5318) (61,5776) (62,5776) (63,5824) (64,5982) (65,6403) (66,6524) (67,6896) (68,6896) (69,6896) (70,7056)

% (0, 2) (1, 10) (2, 32) (3, 32) (4, 82) (5, 82) (6, 87) (7, 135) (8, 170) (9, 200) (10, 247) (11, 277) (12, 361) (13, 387) (14, 387) (15, 454) (16, 521) (17, 637) (18, 637) (19, 706) (20, 745) (21, 796) (22, 875) (23, 919) (24, 973) (25, 1052) (26, 1127) (27, 1205) (28, 1269) (29, 1415) (30, 1564) (31, 1613) (32, 1733) (33, 1898) (34, 1898) (35, 1930) (36, 2121) (37, 2125) (38, 2281) (39, 2466) (40, 2513)

        % (0, 2) (1, 10) (2, 22) (3, 32) (4, 82) (5, 82) (6, 87) (7, 135) (8, 170) (9, 200) (10, 247) (11, 277) (12, 361) (13, 387) (14, 387) (15, 454) (16, 521) (17, 637) (18, 637) (19, 706) (20, 745)
    };
    \addlegendentry{$M_5(d)$}

    \addplot[color=cyan, thick, mark=star] coordinates {

    (0,2) (1,11) (2,21) (3,35) (4,56) (5,64) (6,126) (7,126) (8,138) (9,167) (10,227) (11,268) (12,369) (13,369) (14,390) (15,445) (16,454) (17,495) (18,588) (19,588) (20,654) (21,745) (22,823) (23,831) (24,905) (25,971) (26,1073) (27,1191) (28,1207) (29,1275) (30,1337) (31,1456) (32,1566) (33,1644) (34,1814) (35,1886) (36,1892) (37,2054) (38,2058) (39,2219) (40,2366) (41,2397) (42,2534) (43,2620) (44,2744) (45,2866) (46,2927) (47,3177) (48,3228) (49,3380) (50,3537) (51,3729) (52,3815) (53,3862) (54,4106) (55,4169) (56,4423) (57,4637) (58,4653) (59,4776) (60,4854) (61,5126) (62,5245) (63,5322) (64,5591) (65,5676) (66,5865) (67,6084) (68,6157) (69,6369) (70,6628)

% (0, 2) (1, 11) (2, 21) (3, 35) (4, 56) (5, 64) (6, 126) (7, 126) (8, 138) (9, 167) (10, 227) (11, 268) (12, 369) (13, 369) (14, 390) (15, 445) (16, 454) (17, 495) (18, 588) (19, 588) (20, 654) (21, 745) (22, 823) (23, 831) (24, 905) (25, 971) (26, 1073) (27, 1191) (28, 1207) (29, 1275) (30, 1337) (31, 1456) (32, 1566) (33, 1644) (34, 1814) (35, 1886) (36, 1892) (37, 2054) (38, 2058) (39, 2219) (40, 2366)

        % (0, 2) (1, 11) (2, 21) (3, 35) (4, 56) (5, 64) (6, 126) (7, 126) (8, 138) (9, 167) (10, 227) (11, 268) (12, 369) (13, 369) (14, 390) (15, 445) (16, 454) (17, 495) (18, 588) (19, 588) (20, 654)
    };
    \addlegendentry{$M_6(d)$}

     \addplot[color=violet, thick,mark=star] coordinates {
    
(0,2) (1,14) (2,15) (3,27) (4,59) (5,71) (6,95) (7,146) (8,146) (9,183) (10,258) (11,258) (12,258) (13,320) (14,388) (15,388) (16,438) (17,494) (18,535) (19,595) (20,653) (21,734) (22,771) (23,800) (24,897) (25,978) (26,1045) (27,1107) (28,1174) (29,1242) (30,1244) (31,1502) (32,1608) (33,1644) (34,1747) (35,1747) (36,1768) (37,1918) (38,1951) (39,2035) (40,2246) (41,2376) (42,2376) (43,2565) (44,2595) (45,2754) (46,2825) (47,2927) (48,3086) (49,3193) (50,3411) (51,3411) (52,3539) (53,3660) (54,3761) (55,3933) (56,4122) (57,4147) (58,4390) (59,4602) (60,4629) (61,4822) (62,4882) (63,5127) (64,5319) (65,5502) (66,5557) (67,5665) (68,5842) (69,6134) (70,6268)

% (0, 2) (1, 14) (2, 15) (3, 27) (4, 59) (5, 71) (6, 95) (7, 146) (8, 146) (9, 183) (10, 258) (11, 258) (12, 258) (13, 320) (14, 388) (15, 388) (16, 438) (17, 494) (18, 535) (19, 595) (20, 653) (21, 734) (22, 771) (23, 800) (24, 897) (25, 978) (26, 1045) (27, 1107) (28, 1174) (29, 1242) (30, 1244) (31, 1502) (32, 1608) (33, 1644) (34, 1747) (35, 1747) (36, 1768) (37, 1918) (38, 1951) (39, 2035) (40, 2246)
    };
    \addlegendentry{$M_7(d)$}

 \addplot[color=orange, thick, mark=star] coordinates {

 (0,2) (1,6) (2,17) (3,31) (4,54) (5,64) (6,101) (7,118) (8,137) (9,167) (10,196) (11,215) (12,265) (13,283) (14,336) (15,387) (16,440) (17,467) (18,533) (19,552) (20,633) (21,639) (22,686) (23,758) (24,854) (25,890) (26,931) (27,1012) (28,1089) (29,1259) (30,1280) (31,1322) (32,1379) (33,1509) (34,1587) (35,1620) (36,1693) (37,1776) (38,1941) (39,2015) (40,2047) (41,2139) (42,2309) (43,2388) (44,2493) (45,2685) (46,2854) (47,2901) (48,2929) (49,3070) (50,3134) (51,3320) (52,3510) (53,3596) (54,3609) (55,3815) (56,4011) (57,4121) (58,4147) (59,4326) (60,4506) (61,4686) (62,4844) (63,4899) (64,5037) (65,5232) (66,5309) (67,5501) (68,5682) (69,5797) (70,6098)

  % (0, 2) (1, 6) (2, 17) (3, 31) (4, 54) (5, 64) (6, 101) (7, 118) (8, 137) (9, 167) (10, 196) (11, 215) (12, 265) (13, 283) (14, 336) (15, 387) (16, 440) (17, 467) (18, 533) (19, 552) (20, 633) (21, 639) (22, 686) (23, 758) (24, 854) (25, 890) (26, 931) (27, 1012) (28, 1089) (29, 1259) (30, 1280) (31, 1322) (32, 1379) (33, 1509) (34, 1587) (35, 1620) (36, 1693) (37, 1776) (38, 1941) (39, 2015) (40, 2047)

    };
    \addlegendentry{$M_8(d)$}

     \addplot[color=blue, thick, mark=star] coordinates {
     (0,2) (1,6) (2,13) (3,21) (4,32) (5,45) (6,60) (7,76) (8,94) (9,114) (10,135) (11,159) (12,184) (13,271) (14,272) (15,279) (16,302) (17,334) (18,369) (19,406) (20,444) (21,485) (22,609) (23,611) (24,626) (25,663) (26,710) (27,760) (28,927) (29,928) (30,941) (31,980) (32,1036) (33,1218) (34,1221) (35,1239) (36,1288) (37,1485) (38,1488) (39,1509) (40,1565) (41,1732) (42,1744) (43,1958) (44,1960) (45,2170) (46,2170) (47,2176) (48,2369) (49,2372) (50,2556) (51,2559) (52,2733) (53,2737) (54,2902) (55,2909) (56,3214) (57,3214) (58,3360) (59,3362) (60,3501) (61,3636) (62,3765) (63,3770) (64,3894) (65,4129) (66,4242) (67,4352) (68,4563) (69,4664) (70,4665)

% (0, 2) (1, 6) (2, 13) (3, 21) (4, 32) (5, 45) (6, 60) (7, 76) (8, 94) (9, 114) (10, 135) (11, 159) (12, 184) (13, 271) (14, 272) (15, 279) (16, 302) (17, 334) (18, 369) (19, 406) (20, 444) (21, 485) (22, 609) (23, 611) (24, 626) (25, 663) (26, 710) (27, 760) (28, 927) (29, 928) (30, 941) (31, 980) (32, 1036) (33, 1218) (34, 1221) (35, 1239) (36, 1288) (37, 1485) (38, 1488) (39, 1509) (40, 1565)

    };
    \addlegendentry{$M_{50}(d)$}

%      \addplot[color=black, mark=o] coordinates {
  
% (0, 2) (1, 6) (2, 13) (3, 21) (4, 32) (5, 45) (6, 60) (7, 76) (8, 94) (9, 114) (10, 135) (11, 159) (12, 184) (13, 210) (14, 239) (15, 269) (16, 301) (17, 334) (18, 369) (19, 406) (20, 444) (21, 485) (22, 526) (23, 570) (24, 615) (25, 662) (26, 710) (27, 760) (28, 927) (29, 928) (30, 941) (31, 980) (32, 1036) (33, 1096) (34, 1158) (35, 1221) (36, 1286) (37, 1353) (38, 1421) (39, 1491) (40, 1562)

%     };
%     \addlegendentry{$M_{100}(d)$}

    \end{axis}
\end{tikzpicture}
\caption{Growth of $M_k(d)$ with $d\leq 10^{70}$}
\end{figure}
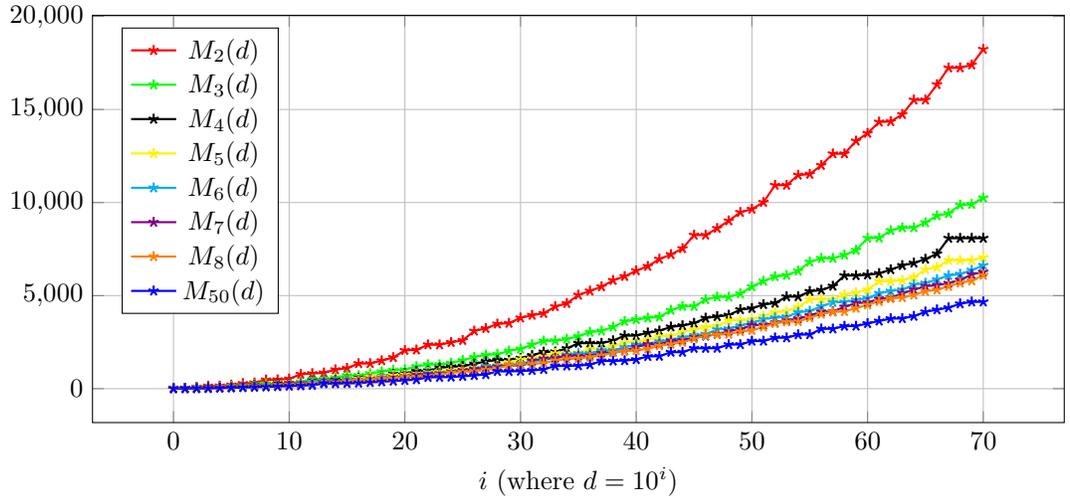
\smallskip

It is tempting to seek a model that asymptotically gives the values $M_k(d)$ as $d\rightarrow +\infty.$ A natural approach is to
apply the standard non-linear least squares fitting method. We did so using the values at powers of 10 with $d\leq N,$ and we computed approximations $f_{k}(N; d)$ that are polynomials in $\log(d).$  For $k=50$ and $N\in \{10^{12}, 10^{70}\}$, we obtain two very different approximations
\begin{displaymath}
\begin{split}
f_{50}(10^{12};d) & \approx 2.018\times10^{-13} \cdot (\log d)^3 + 0.174 \cdot (\log d)^2 + 1.836 \cdot \log d + 0.748,\\
f_{50}(10^{70}; d) &\approx 1.91\times10^{-8} \cdot \log(d)^5 - 8.77\times10^{-6} \cdot \log(d)^4 + 1.25\times10^{-3} \cdot \log(d)^3 \\
&\ \ \hskip1.75in + 0.1118\cdot\log(d)^2+2.260\cdot\log(d)+8.946.
\end{split}
\end{displaymath}
It is natural to speculate whether the $M_k(d)$ are asymptotically polynomials in $\log(d)$ as $d\rightarrow \infty$. In Figure 2 we juxtapose some of the $M_k(d)$ in Table~2 with $\log(d)$ and $\log(d)^2/5$. The trends in this figure suggest that these values are growing much faster than any scalar multiple of $\log(d)$, but do not preclude the possibility that they are asymptotically higher degree (perhaps quadratic or cubic) polynomials in $\log(d).$

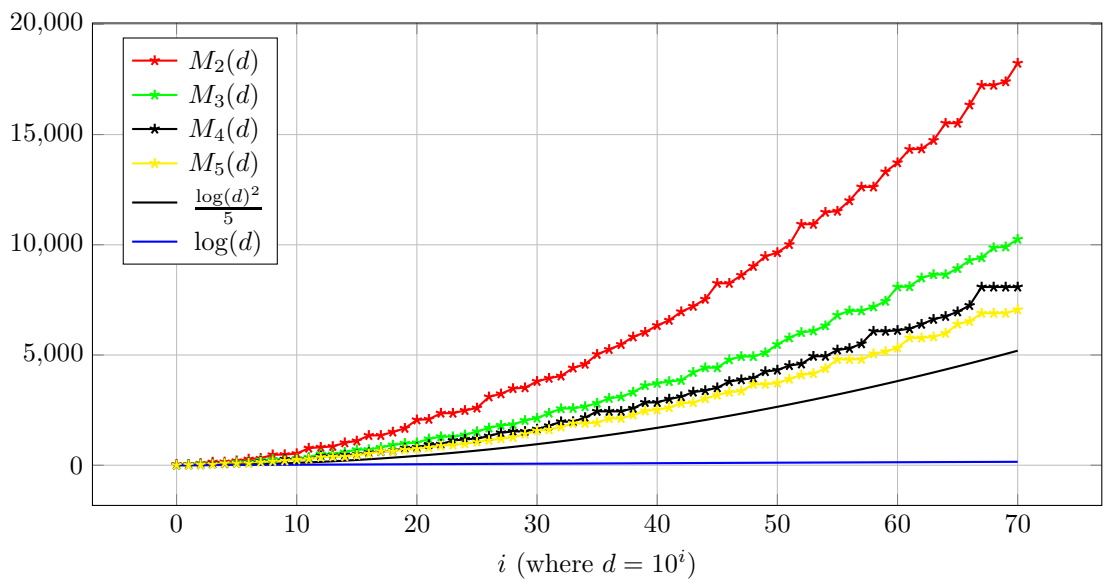
\begin{figure}[H]
\begin{tikzpicture}
    \begin{axis}[
        width=15cm,
        height=8cm,
        xlabel={$i$ (where $d = 10^i$)},
        ylabel={\ \ },
        legend pos=north west,
        grid=major,
        scaled ticks=false,
    ]

    % M_k(d) values
    \addplot[color=red, thick, mark=star] coordinates {
(0,35) (1,35) (2,87) (3,143) (4,148) (5,201) (6,280) (7,312) (8,473) (9,483) (10,527) (11, 775) (12, 815)(13, 851) (14, 1020) (15, 1100)(16, 1352)(17, 1352) (18, 1505)(19,1675) (20, 2058)(21, 2084) (22, 2361) (23, 2361) (24, 2487) (25, 2595)(26, 3101)(27, 3238) (28, 3479) (29, 3517) (30, 3804)(31,3947) (32,4047) (33,4412) (34,4582) (35,5030) (36,5253) (37,5479) (38,5822) (39,6032) (40,6340) (41,6576) (42,6964) (43,7210) (44,7530) (45,8253) (46,8253) (47,8615) (48,9021) (49,9475) (50,9646)(51,10013) (52,10932) (53,10932) (54,11477) (55,11524) (56,12000) (57,12623) (58,12623) (59,13309) (60,13723) (61,14329) (62,14352) (63,14736) (64,15516) (65,15516) (66,16342) (67,17239) (68,17239) (69,17399) (70,18237)

     % (0, 35) (1, 35) (2, 87) (3, 143) (4, 148) (5, 201) (6, 280) (7, 312) (8, 473) (9, 483) (10, 527) (11, 775) (12, 815) (13, 851) (14, 1020) (15, 1100) (16, 1352) (17, 1352) (18, 1505) (19, 1675) (20, 2058) (21, 2084) (22, 2361) (23, 2361) (24, 2487) (25, 2595) (26, 3101) (27, 3238) (28, 3479) (29, 3517) (30, 3804) (31, 3947) (32, 4047) (33, 4412) (34, 4582) (35, 5030) (36, 5253) (37, 5479) (38, 5822) (39, 6032) (40, 6340) (45, 8253) (50, 9646) (55, 11524)
        % (0, 35) (1, 35) (2, 87) (3, 143) (4, 148) (5, 201) (6, 280) (7, 312) (8, 473) (9, 483) (10, 527) (11, 775) (12, 815) (13, 851) (14, 1020) (15, 1100) (16, 1352) (17, 1352) (18, 1505) (19, 1675) (20, 2058)
    };
    \addlegendentry{$M_2(d)$}

    \addplot[color=green, thick, mark=star] coordinates {

(0,5) (1,22) (2,57) (3,57) (4,75) (5,133) (6,146) (7,229) (8,261) (9,270) (10,295) (11, 361) (12, 505)(13, 505) (14,592) (15,705)(16,709) (17, 811) (18, 914) (19, 995) (20, 1019)(21, 1194) (22, 1294) (23, 1294) (24, 1376) (25, 1525)(26, 1693) (27, 1812) (28, 1848) (29, 2032) (30, 2135)(31,2364) (32,2566) (33,2579) (34,2661) (35,2815) (36,3035) (37,3102) (38,3315) (39,3616) (40,3714) (41,3800) (42,3851) (43,4214) (44,4424) (45,4424) (46,4786) (47,4931) (48,4931) (49,5105) (50,5479)(51,5779) (52,6023) (53,6093) (54,6327) (55,6808) (56,7005) (57,7005) (58,7184) (59,7440) (60,8088) (61,8102) (62,8499) (63,8650) (64,8650) (65,8924) (66,9292) (67,9413) (68,9870) (69,9907) (70,10252)
% (0, 5) (1, 22) (2, 57) (3, 57) (4, 75) (5, 133) (6, 146) (7, 229) (8, 261) (9, 270) (10, 295) (11, 361) (12, 505) (13, 505) (14, 592) (15, 705) (16, 709) (17, 811) (18, 914) (19, 995) (20, 1019) (21, 1194) (22, 1294) (23, 1294) (24, 1376) (25, 1525) (26, 1693) (27, 1812) (28, 1848) (29, 2032) (30, 2135) (31, 2364) (32, 2566) (33, 2579) (34, 2661) (35, 2815) (36, 3035) (37, 3102) (38, 3315) (39, 3616) (40, 3714) 

        % (0, 5) (1, 22) (2, 57) (3, 57) (4, 75) (5, 133) (6, 146) (7, 229) (8, 261) (9, 270) (10, 295) (11, 361) (12, 505) (13, 505) (14, 592) (15, 705) (16, 709) (17, 811) (18, 914) (19, 995) (20, 1019)
    };
    \addlegendentry{$M_3(d)$}

    \addplot[color=black, thick, mark=star] coordinates {

(0,7) (1,20) (2,26) (3,57) (4,79) (5,87) (6,110) (7,170) (8,170) (9,265) (10,265) (11, 287) (12, 420)(13, 420) (14, 463) (15, 482)(16, 571) (17, 627) (18, 649) (19, 773) (20, 806) (21, 938) (22, 952) (23, 1120) (24, 1174) (25, 1203)(26, 1325) (27, 1470) (28, 1534) (29, 1534) (30, 1636)(31,1789) (32,1973) (33,1978) (34,2162) (35,2444) (36,2444) (37,2444) (38,2582) (39,2849) (40,2849) (41,3001) (42,3107) (43,3320) (44,3387) (45,3516) (46,3795) (47,3879) (48,3952) (49,4250) (50,4314)(51,4525) (52,4595) (53,4943) (54,4943) (55,5229) (56,5299) (57,5511) (58,6078) (59,6078) (60,6117) (61,6192) (62,6395) (63,6619) (64,6752) (65,6961) (66,7248) (67,8084) (68,8084) (69,8084) (70,8084)
    % (0, 7) (1, 20) (2, 26) (3, 57) (4, 79) (5, 87) (6, 110) (7, 170) (8, 170) (9, 265) (10, 265) (11, 287) (12, 420) (13, 420) (14, 463) (15, 482) (16, 571) (17, 627) (18, 649) (19, 773) (20, 806) (21, 938) (22, 952) (23, 1120) (24, 1174) (25, 1203) (26, 1325) (27, 1470) (28, 1534) (29, 1534) (30, 1636) (31, 1789) (32, 1973) (33, 1978) (34, 2162) (35, 2444) (36, 2444) (37, 2444) (38, 2582) (39, 2849) (40, 2849)

        % (0, 7) (1, 20) (2, 26) (3, 57) (4, 79) (5, 87) (6, 110) (7, 170) (8, 170) (9, 265) (10, 265) (11, 287) (12, 420) (13, 420) (14, 463) (15, 482) (16, 571) (17, 627) (18, 649) (19, 773) (20, 806)
    };
    \addlegendentry{$M_4(d)$}

    \addplot[color=yellow, thick, mark=star] coordinates {

(0,2) (1,10) (2,22) (3,32) (4,82) (5,82) (6,87) (7,135) (8,170) (9,200) (10,247) (11, 277) (12, 361)(13, 387) (14, 387) (15, 454)(16, 521) (17, 637) (18, 637) (19, 706)(20, 745)(21, 796) (22, 875) (23, 919) (24, 973) (25, 1052)(26, 1127) (27, 1205) (28, 1269) (29, 1415) (30, 1564)(31,1613) (32,1733) (33,1898) (34,1898) (35,1930) (36,2121) (37,2125) (38,2281) (39,2466) (40,2513) (41,2613) (42,2807) (43,2841) (44,3023) (45,3178) (46,3313) (47,3365) (48,3671) (49,3671) (50,3726)(51,3907) (52,4100) (53,4159) (54,4386) (55,4802) (56,4802) (57,4802) (58,5054) (59,5155) (60,5318) (61,5776) (62,5776) (63,5824) (64,5982) (65,6403) (66,6524) (67,6896) (68,6896) (69,6896) (70,7056)
% (0, 2) (1, 10) (2, 32) (3, 32) (4, 82) (5, 82) (6, 87) (7, 135) (8, 170) (9, 200) (10, 247) (11, 277) (12, 361) (13, 387) (14, 387) (15, 454) (16, 521) (17, 637) (18, 637) (19, 706) (20, 745) (21, 796) (22, 875) (23, 919) (24, 973) (25, 1052) (26, 1127) (27, 1205) (28, 1269) (29, 1415) (30, 1564) (31, 1613) (32, 1733) (33, 1898) (34, 1898) (35, 1930) (36, 2121) (37, 2125) (38, 2281) (39, 2466) (40, 2513)

        % (0, 2) (1, 10) (2, 22) (3, 32) (4, 82) (5, 82) (6, 87) (7, 135) (8, 170) (9, 200) (10, 247) (11, 277) (12, 361) (13, 387) (14, 387) (15, 454) (16, 521) (17, 637) (18, 637) (19, 706) (20, 745)
    };
    \addlegendentry{$M_5(d)$}
    
        \addplot[color=black, thick, mark=] coordinates {
    
    (0,0) (1, 1.06) (2, 4.2) (3,9.5) (4, 16.9 ) (5, 26.4) (6, 38.0) (7, 51.8) (8, 65.5) (9, 82.4) (10, 106.0) (11, 128.3) (12, 152.6)(13, 179.2)(14, 207.8) (15,238.5)(16, 271.4) (17, 306.4) (18, 343.5) (19, 382.7)(20,424.1)(21, 467.6) (22, 513.2) (23, 560.9) (24, 610.7) (25, 662.7)(26, 716.8) (27, 773.0) (28, 831.3) (29, 891.7) (30, 954.3)(31, 1019.0)(32, 1085.8) (33, 1154.7) (34, 1225.7)(35, 1298.9) (36, 1374.2) (37, 1451.6) (38, 1531.8) (39, 1612.8) (40, 1696.6) (41, 1782.4)(42, 1870.5)(43, 1960.6) (44, 2052.8) (45, 2147.2)(46, 2243.7) (47, 2342.3) (48, 2443.1) (49, 2545.9) (50, 2650.9)(51,2758)(52,2868.2)(53,2978.6)(54, 3092.0)(55, 3207.6)(56,3325.3)(57, 3445.1)(58, 3567.1)(59,3691.1)(60, 3817.3)(61, 3945.6)(62, 4076.1)(63,4208.6)(64, 4343.3)(65,4480.1)(66, 4619.0)(67,4760.0)(68,4903.2)(69, 5048.4)(70,5195.8)
            % (0, 2) (1, 11) (2, 21) (3, 35) (4, 56) (5, 64) (6, 126) (7, 126) (8, 138) (9, 167) (10, 227) (11, 268) (12, 369) (13, 369) (14, 390) (15, 445) (16, 454) (17, 495) (18, 588) (19, 588) (20, 654) (21, 745) (22, 823) (23, 831) (24, 905) (25, 971) (26, 1073) (27, 1191) (28, 1207) (29, 1275) (30, 1337) (31, 1456) (32, 1566) (33, 1644) (34, 1814) (35, 1886) (36, 1892) (37, 2054) (38, 2058) (39, 2219) (40, 2366)

        % (0, 2) (1, 11) (2, 21) (3, 35) (4, 56) (5, 64) (6, 126) (7, 126) (8, 138) (9, 167) (10, 227) (11, 268) (12, 369) (13, 369) (14, 390) (15, 445) (16, 454) (17, 495) (18, 588) (19, 588) (20, 654)
    };
    \addlegendentry{$\ \frac{\log(d)^2}{5}$}

    \addplot[color=blue, thick, mark=] coordinates {
    
    (0,0) (1,2.30) (2,4.6) (3,6.9) (4,9.2) (5,11.5) (6,13.8) (7,16.1) (8,18.4) (9,20.72) (10,23.0) (11, 25.3)(12, 27.6)(13, 29.9)(14, 32.2)(15, 34.5)(16, 36.8) (17, 39.1)(18, 41.4)(19, 43.7)(20, 46.0)(21, 48.3) (22, 50.6) (23, 52.9) (24, 55.2) (25, 57.5)(26, 59.8) (27, 62.1) (28, 64.4) (29, 66.7) (30, 69.0)(31, 71.3)(32, 73.6) (33, 75.9)(34, 78.2)(35, 80.5)(36, 82.8)(37, 85.1)(38, 87.4)(39, 89.8)(40, 92.1)(41, 94.4)(42, 96.7)(43, 99.0)(44, 101.3)(45, 103.6)(46, 105.9)(47, 108.2)(48, 110.5)(49, 112.8)(50, 115.1)(51,117)(52,119)(53,121)(54,123)(55,125)(56,127)(57,129)(58,131)(59,133)(60,135)(61,137)(62,139)(63,141)(64,143)(65,145)(66,147)(67,149)(68,151)(69,153)(70,155)
            % (0, 2) (1, 11) (2, 21) (3, 35) (4, 56) (5, 64) (6, 126) (7, 126) (8, 138) (9, 167) (10, 227) (11, 268) (12, 369) (13, 369) (14, 390) (15, 445) (16, 454) (17, 495) (18, 588) (19, 588) (20, 654) (21, 745) (22, 823) (23, 831) (24, 905) (25, 971) (26, 1073) (27, 1191) (28, 1207) (29, 1275) (30, 1337) (31, 1456) (32, 1566) (33, 1644) (34, 1814) (35, 1886) (36, 1892) (37, 2054) (38, 2058) (39, 2219) (40, 2366)

        % (0, 2) (1, 11) (2, 21) (3, 35) (4, 56) (5, 64) (6, 126) (7, 126) (8, 138) (9, 167) (10, 227) (11, 268) (12, 369) (13, 369) (14, 390) (15, 445) (16, 454) (17, 495) (18, 588) (19, 588) (20, 654)
    };
    \addlegendentry{$\ \log(d)$}

 \comment    \addplot[color=violet, thick,mark=] coordinates {
    
(0,0) (1,2.30) (2,4.6) (3,6.9) (4,9.2) (5,11.5) (6,13.8) (7,16.1) (8,18.4) (9,20.72) (10,23.0) (11,25.3) (12,27.6) (13,29.9) (14,32.2) (15,34.5) (16,36.8) (17,39.1) (18,41.4) (19,43.7) (20,46) (21,48.3) (22,50.6) (23,52.9) (24, 55.2) (25, 57.5) (26, 59.8) (27, 62.1) (28, 64.4) (29, 66.7) (30, 69) (31,71.3) (32,73.6) (33, 75.9) (34, 78.2) (35, 80.5) (36, 82.8) (37, 85.1) (38,87.4) (39, 89.8) (40, 92.1) (41, 94.4) (42, 96.7) (43, 99) (44, 101.3) (45, 103.6) (46, 105.9) (47, 108.2) (48, 110.5) (49, 112.8) (50, 115.1) (51, 117.4) (52, 119.7) (53, 122) (54, 124.3) (55, 126.6) (56, 128.9) (57, 131.2) (58, 133.5) (59, 135.8) (60, 138.1) (61, 140.4) (62, 142.7) (63, 145) (64, 147.3) (65, 149.6) (66, 151.9) (67, 154.2) (68, 156.5) (69, 158.8) (70, 161.1)

% (0, 2) (1, 14) (2, 15) (3, 27) (4, 59) (5, 71) (6, 95) (7, 146) (8, 146) (9, 183) (10, 258) (11, 258) (12, 258) (13, 320) (14, 388) (15, 388) (16, 438) (17, 494) (18, 535) (19, 595) (20, 653) (21, 734) (22, 771) (23, 800) (24, 897) (25, 978) (26, 1045) (27, 1107) (28, 1174) (29, 1242) (30, 1244) (31, 1502) (32, 1608) (33, 1644) (34, 1747) (35, 1747) (36, 1768) (37, 1918) (38, 1951) (39, 2035) (40, 2246)
    };
    \addlegendentry{$\ \ \log(d)^2$}
\end{comment}

    \end{axis}
\end{tikzpicture}
\caption{Growth of $M_k(d)$ with $k\in \{2, 3, 4, 5\}$}
\end{figure}
\smallskip

\noindent
It would be very interesting to carry out even further numerics to address the following problem.

\begin{problem}  For each $k,$ conjecture an asymptotic formula for the $M_k(d).$ 
\end{problem}

To conclude, we offer a curious conjecture about the sequences $\{M_1(d), M_2(d), M_3(d),\dots\}$ for each fixed $d$.
The first two rows of Table 2 suggest that these sequences decrease and stabilize at the values 1 and 2 respectively. We have done many further numerics, and we believe that each of these sequences eventually stabilizes. Table~3 below suggests that these limiting values are the positive integers in order repeated with curious multiplicities.

\begin{table}[H]
	\begin{tabular}{|c|c|c|c|c|c|c|c|c|c|c|c|c|}\hline
	$d$ & $M_2(d)$ & $M_3(d)$  & $M_4(d)$ & $M_5(d)$ & $M_6(d)$ & $M_7(d)$ & $M_8(d)$ & $\cdots$ & $M_{50}(d)$ & $\cdots$ & $M_{100}(d)$ \\ \hline
	$0$ & $1$ & $1$ & $1$ & $1$ & $1$ & $1$ & $1$ & $\cdots$ & $1$ & $\cdots$ & $1$ \\ \hline
	$1$ & $ 35$ & $ 5$ & $ 7$ & $ 2$ & $ 2$ & $ 2$ & $ 2$ & $\cdots$ & $ 2$ & $\cdots$ & $ 2$ \\ \hline
	$2$ & $ 35$ & $ 22$ & $ 20$ & $ 9$ & $ 3$ & $ 3$ & $ 3$ & $\cdots$ & $ 3$ & $\cdots$ & $ 3$ \\ \hline
	$3$ & $ 35$ & $ 22$ & $ 20$ & $ 9$ & $ 3$ & $ 3$ & $ 3$ & $\cdots$ & $ 3$ & $\cdots$ & $ 3$ \\ \hline
	$4$ & $ 35$ & $ 22$ & $ 20$ & $ 9$ & $ 4$ & $ 4$ & $ 4$ & $\cdots$ & $ 4$ & $\cdots$ & $4$ \\ \hline
	$5$ & $ 35$ & $ 22$ & $ 20$ & $ 9$ & $ 4$ & $ 4$ & $ 4$ & $\cdots$ & $ 4$ & $\cdots$ & $ 4$ \\ \hline
    $6$ & $ 35$ & $ 22$ & $ 20$ & $ 9$ & $ 5$ & $ 5$ & $ 5$ & $\cdots$ & $5$ & $\cdots$ & $ 5$ \\ \hline
	\end{tabular}\smallskip	
	\caption{Conjectured values of $M_k(d)$}
\end{table}
\normalsize

\begin{conjecture} For each fixed non-negative integer $d$, there is an integer $L_d$ for which
$$
M_k(d)=L_d
$$
for all sufficiently large $k.$
\end{conjecture}

We refine this claim  by offering a conjectural partition theoretic description of the limiting values and their multiplicities,
which is based on the number of partitions of $n$ without parts of size 1, denoted by $p_1(n).$   Clearly, the generating function for $p_1(n)$ is
$$
\sum_{n=0}^{\infty}p_1(n)q^n=\prod_{n=2}^{\infty}\frac{1}{1-q^n}=1+q^2+q^3+2q^4+2q^5+4q^6+4q^7+7q^8+8q^9+12q^{10}+\dots.
$$
In terms of $p_1(n)$, we define the
the auxiliary partition 
\begin{equation}
\psi(n) = \sum_{j=1}^{n} p_1(j),
\end{equation}
the number of non-empty partitions of size $\leq n$ without parts of size 1. It is not difficult to confirm that its generating function is given by 
$$
\Psi(q)=\sum_{n=1}^{\infty}\psi(n)q^n=
\sum_{n=2}^{\infty}\frac{q^n}{(1-q)(1-q^2)\cdots (1-q^n)}=q^2+2q^3+4q^4+6q^5+10q^6+14q^7+21q^8+29q^9+41q^{10}+\dots.
$$
The astute reader will notice that $\psi(n)=p(n)-1$ for all $n.$
In loose analogy with the $M_k(d),$ we define the function
\begin{equation}
L(d):= \max \left\{ n \ : \ \psi(n)\leq d\right\} =\max \left\{ n \ : \ p(n)-1\leq d\right\}.
\end{equation}
Clearly, we have that
$$
L(d) = 
\begin{cases} 
\ 1 & \text{for } d = 0, \\
\ 2 & \text{for } d = 1, \\
\ 3 & \text{for } d = 2, 3, \\
\ 4 & \text{for } d = 4, 5, \\
\ 5 & \text{for } d = 6, 7, 8, 9, \\
\ \vdots & \\
\ n & \text{repeated } p_1(n+1) \text{ times}.
\end{cases}
$$
We believe that the values of $L(d)$ are the limiting values $L_d$ which make up the alleged curious sequence of all the positive integers in order with curious multiplicities, a claim which essentially means that  $1$ is generally the closest $k$th power among the partition numbers.

\begin{conjecture}
For each non-negative integer $d$, there is a positive integer $N_d$ such that for $k\geq N_d$ we have
$$
M_k(d) = L(d).
$$
\end{conjecture}

\begin{remark}In Table~4 below, we give the conjectured minimal values that can be taken for $N_d$ in Conjecture~4.
We believe that $N_d=O(\log(d)).$ 

\smallskip

\begin{table}[h]
	% \caption{\label{table:1} Conjectured minimal values of $N(d)$}
	\begin{tabular}{|r|c|c|c|c|c|c|c|c}
		\hline %\rule[-3mm]{0mm}{8mm}
		$d$    & 0 & 1 & [2,6]  & [7,21] & [22,89] & [90,156] & [157,1500]   \\   \hline
		$N_d$ & 2 & 5 & 6 & 8 & 11 & 12 & 14  \\
		\hline
		\hline %\rule[-3mm]{0mm}{8mm}
		$d$ & [1501,1582]  & [1583,2274] & [2275,2534] & [2535,72928] & [72929,84593] & [84594,106335] & [106336,270343]  \\   \hline
		$N_d$ & 15 & 16 & 17 & 20 & 21 & 22 & 23   \\
		\hline
\end{tabular}
\caption{Conjectured minimal values for $N_d$ in Conjecture 4} \end{table} 
 
\end{remark}  
     
\section*{Concluding thoughts}

Conjectures about the asymptotics and congruence properties of $p(n)$ have driven research in analytic number theory and the theory of modular forms. Indeed, the partition numbers gave birth to the ``circle method'' in analytic number theory through the work of Hardy and Ramanujan, and has inspired the development of Hecke operators and the theory of modular forms congruences through the work of Atkin and many others. It is our hope that the conjectures presented here will inspire further deep advances in the field.

 \section*{Acknowledgements} 
 \noindent
 We note that all of the calculations required for this paper were performed on a computer running Wolfram Mathematica 12.0.

\end{document}